\documentclass[11pt]{article}
\usepackage[a4paper,margin=1in]{geometry}
\usepackage[T1]{fontenc}
\usepackage[utf8]{inputenc}
\usepackage{lmodern}
\usepackage{amsmath,amssymb,amsthm,mathtools}
\usepackage{bm}
\usepackage{booktabs}
\usepackage{array}
\usepackage{tikz}
\usetikzlibrary{arrows.meta,positioning,calc}
\usepackage{pgfplots}
\pgfplotsset{compat=1.18}
\usepackage{enumitem}
\usepackage{microtype}
\usepackage[numbers,sort&compress]{natbib}
\usepackage[colorlinks=true,allcolors=blue]{hyperref}
\usepackage[nameinlink,noabbrev]{cleveref}
\newtheorem{theorem}{Theorem}[section]
\newtheorem{proposition}[theorem]{Proposition}
\newtheorem{corollary}[theorem]{Corollary}
\newtheorem{lemma}[theorem]{Lemma}
\newtheorem{remark}[theorem]{Remark}
\newtheorem{example}[theorem]{Example}
\theoremstyle{definition}
\newtheorem{definition}[theorem]{Definition}

\newcommand{\R}{\mathbb R}

\newcommand{\mug}{\mu_g}
\newcommand{\mugc}{\mu_g^{\mathrm c}}
\newcommand{\Dg}{D_g}
\newcommand{\ACg}{AC_g}
\newcommand{\Ran}{\operatorname{Ran}}
\newcommand{\Ker}{\operatorname{Ker}}
\newcommand{\Span}{\operatorname{span}}
\newcommand{\dist}{\operatorname{dist}}
\newcommand{\norm}[1]{\left\lVert #1\right\rVert}
\newcommand{\abs}[1]{\left|#1\right|}

\newcommand{\cR}{\mathcal R}
\newcommand{\cK}{\mathcal K}

\newcommand{\cB}{\mathcal B}
\newcommand{\cU}{\mathcal U}
\newcommand{\Lam}{\Lambda}
\newcommand{\eps}{\varepsilon}
\newcommand{\Rpre}{\mathcal R_-}
\newcommand{\Rpost}{\mathcal R_+}
\newcommand{\Rdel}{\mathcal R_{\rm del}}

\title{Vanishing Events and Regularity-Conditioned Reachability\\in Linear Stieltjes Systems}
\author{Francisco J. Fern\'andez\\[0.6ex]
\small\textit{Departamento de Estat\'istica, An\'alise Matem\'atica e Optimizaci\'on,}\\
\small\textit{Facultade de Matem\'aticas, Campus Vida, Universidade de Santiago de Compostela,}\\
\small\textit{15782 Santiago de Compostela, Spain}\\
\small\textit{CITMAga, 15782 Santiago de Compostela, Spain}\\[0.5ex]
\small Email: \href{mailto:fjavier.fernandez@usc.es}{fjavier.fernandez@usc.es}\\
\small ORCID: \href{https://orcid.org/0000-0002-0667-2639}{0000-0002-0667-2639}}
\date{}

\begin{document}
\maketitle

\begin{abstract}
We study linear Stieltjes control systems in which actuation is concentrated at finitely many temporal events, and we ask what happens when one of those events progressively loses its atomic weight in the associated Stieltjes measure. For constant coefficients, ordinary finite-atomic reachability is governed by a truncated Krylov space and is therefore insensitive to the positive values of the event weights. A different picture emerges once trajectory regularity is imposed. In the supercritical super-$g$-H\"older regime, we minimize the largest event-normalized state jump and characterize exactly the terminal directions for which this cost remains uniformly bounded as a prescribed event weight tends to zero. The characterization combines what can be reached before the vulnerable event, a drift-cancellability condition at that event, and the directions that remain available afterwards. Ordinary reachability of the deleted-event system can therefore be strictly larger than bounded-cost reachability in the singular limit. We also separate two mechanisms of divergence: loss of ordinary reachability after deletion and failure of the intermediate compatibility condition. For compatible targets, the minimum cost converges to a constrained deleted-event problem. A dimensionless position--velocity benchmark motivated by pulsed proximity maneuvers shows the distinction explicitly.
\end{abstract}



\section{Introduction}\label{sec:intro}

Consider a controlled system in which useful actuation is available only at a finite sequence of temporal opportunities.  Suppose that one of these opportunities becomes progressively weaker.  The first question is classical: can the same terminal state still be reached?  But this is not the only relevant question.  Even if endpoint reachability survives, the trajectory may be forced to use the shrinking event in an increasingly singular way.  The purpose of this paper is to separate these two effects and to identify exactly which terminal directions remain reachable with a uniformly bounded regularity cost.

Pulsed and impulsive control provide a natural motivation.  In spacecraft proximity operations, for example, relative motion is described in position--velocity variables, maneuver plans use discrete corrections, and minimum-impulse or execution constraints limit how accurately an individual correction can be realized \citep{ClohessyWiltshire1960,Dannemiller2011,PeiEtAl2017,BotelhoRosaLemos2025}.  Robust rendezvous studies also account for execution errors and partial or missed thrust \citep{LouembetArzelierDeaconu2015,AguilarMarsillachDiCairanoWeiss2020,NemirovskyGurfil2026}.  These works motivate the problem class.  The model below is deliberately more abstract: it is a Stieltjes event model, not a calibrated spacecraft propagator.

Stieltjes dynamics gives a convenient way to encode heterogeneous temporal behaviour.  A nondecreasing left-continuous derivator $g$ may evolve continuously, remain flat on inactive intervals, and jump at isolated events.  Its associated measure $\mu_g$ records all three behaviours.  In a finite-atomic clock, the positive atomic weights $m_1,\ldots,m_N$ measure the operational weight of the events.  They are not, by themselves, burn durations, delivered impulses, energies, or fault fractions.

For constant matrices $A$ and $B$, the ordinary reachability of a purely atomic Stieltjes system with $N$ positive events is already known.  From the zero state it is the truncated Krylov space
\begin{equation}\label{eq:intro-krylov}
 \mathcal K_N(A,B)
 =\operatorname{span}\{\operatorname{Ran}B,\operatorname{Ran}AB,\ldots,
 \operatorname{Ran}A^{N-1}B\}
\end{equation}
\citep{FernandezReachability2026}.  The crucial point for the present paper is not the formula itself but what it does \emph{not} see: as long as every event weight remains positive, the reachable subspace depends on the number of events, not on their positive values.  An event can therefore become arbitrarily small while the ordinary Kalman/Krylov geometry remains unchanged.

Regularity relative to the Stieltjes clock detects what ordinary reachability misses.  The super-$g$-H\"older theory of \citet{IlterDuruKoca2026} shows that, for $\alpha>1$, all state variation is atomic and the jump carried by an event of weight $m_j$ is measured through
\[
 \frac{\|\Delta x_j\|}{m_j^\alpha}.
\]
This is a trajectory-regularity quantity rather than a fuel or energy cost.  A fixed jump becomes more expensive when it is assigned to a smaller event.  The singular question is therefore precise: if one event weight $m_\ell=\varepsilon$ tends to zero while all other event weights stay fixed and positive, for which terminal directions can these normalized jumps remain uniformly bounded?

The mechanism is already visible at the vanishing event.  Write
\[
 w_\ell=\frac{x_\ell-x_{\ell-1}}{\varepsilon^\alpha}.
\]
The event equation becomes
\begin{equation}\label{eq:intro-hidden}
 \varepsilon^{\alpha-1}w_\ell=Ax_{\ell-1}+Bu_\ell.
\end{equation}
Let $Q$ be the orthogonal projection onto $(\operatorname{Ran}B)^\perp$.  Projecting \eqref{eq:intro-hidden} removes the control term:
\[
 QA x_{\ell-1}=\varepsilon^{\alpha-1}Qw_\ell.
\]
If the normalized jump stays bounded and $\alpha>1$, the right-hand side tends to zero.  Any limiting pre-event state must therefore satisfy
\begin{equation}\label{eq:intro-compat}
 Ax_{\ell-1}\in\operatorname{Ran}B.
\end{equation}
Thus the vanishing event leaves behind an intermediate compatibility condition: the drift present just before the event must lie in a direction that the control can cancel.

The main theorem propagates this local condition to the terminal state.  Let $\mathcal B_\ell$ be the set of terminal directions whose minimum super-$g$-H\"older cost stays uniformly bounded as $\varepsilon\downarrow0$.  Then
\begin{equation}\label{eq:intro-main}
 \mathcal B_\ell
 =\bigl(\mathcal K_{\ell-1}(A,B)\cap A^{-1}(\operatorname{Ran}B)\bigr)
  +\mathcal K_{N-\ell}(A,B).
\end{equation}
The formula has a direct reading.  Before the vulnerable event, the system may build any state in the corresponding Krylov space, but only those pre-event states whose drift is cancellable survive the singular limit.  After the event, the remaining controls add the directions available in the post-event Krylov space.  The resulting \emph{subspace} is independent of the remaining positive event weights, although the numerical value of the minimum cost need not be.

This bounded-cost geometry can be strictly smaller than ordinary reachability after deletion.  For the double integrator
\[
 A=\begin{pmatrix}0&1\\0&0\end{pmatrix},\qquad B=e_2,
\]
with three event weights $(a,\varepsilon,b)$ and the middle one vanishing, deleting the middle event leaves an ordinarily controllable two-event system, $\mathcal R_{\rm del}=\mathbb R^2$.  By contrast, \eqref{eq:intro-main} gives $\mathcal B_2=\operatorname{span}\{e_2\}$.  A normalized terminal position correction $e_1$ remains reachable, but its minimum regularity cost diverges as $\varepsilon\downarrow0$.  The raw state jump at the vulnerable event still tends to zero; what becomes singular is the jump relative to the stronger scale $\varepsilon^\alpha$.

The same distinction leads to two different blow-up mechanisms.  If a target leaves the reachable space when the event is deleted, then the shrinking event must supply a terminal component that the remaining system cannot generate; this produces a range-loss lower bound of order $\varepsilon^{-\alpha}$.  A subtler case occurs when the target remains ordinarily reachable after deletion, but every limiting realization violates \eqref{eq:intro-compat}.  This compatibility loss produces a lower-bound scale $\varepsilon^{1-\alpha}$.  For targets in $\mathcal B_\ell$, the minimum costs instead converge to a deleted-event problem with the intermediate compatibility condition imposed explicitly.

The ingredients around this result come from several established theories.  Ordinary Stieltjes controllability and the finite-atomic Krylov formula are available in \citet{SiWangFeckan2020,FernandezReachability2026}; the supercritical atomic representation is due to \citet{IlterDuruKoca2026}; and the finite-dimensional optimization is a classical minimum-amplitude problem \citep{Cadzow1974,Kolev1975}.  Predecessor constructions and inverse-image recursions are standard in geometric and constrained control \citep{BasileMarro1969,Wonham1979,BorrelliBemporadMorari2017}, while singularly parameterized optimization has a broader variational theory \citep{AvrachenkovEtAl2012,RockafellarWets1998}.  We do not claim novelty for these ingredients separately.  The contribution here is the exact fixed-target characterization selected by the vanishing-event supercritical scaling, together with the distinction between range loss and compatibility loss.

The rest of the paper follows the logic of this argument.  \Cref{sec:framework} fixes the Stieltjes conventions and recalls ordinary finite-atomic reachability.  \Cref{sec:regularity} explains the supercritical regularity filter, and \Cref{sec:atomic} rewrites the problem in normalized-jump variables.  \Cref{sec:vanishing} proves the exact bounded-cost subspace theorem.  \Cref{sec:costlimit} studies the limiting cost and the two blow-up mechanisms.  \Cref{sec:application} returns to the position--velocity benchmark, and \Cref{sec:discussion} summarizes the scope and consequences.
\section{Stieltjes framework and ordinary reachability}\label{sec:framework}

\subsection{The Stieltjes clock and why atoms matter}

The derivator $g$ is the clock of the system.  We use the standard first-order Stieltjes framework of \citet{FrigonPouso2017}.  Its increments, rather than Euclidean time increments, measure how much evolution is available between two physical times.  We now fix the conventions needed later, paying particular attention to atoms because the main singular limit concerns the disappearance of one such event.

Let $I=[a,b]$, $a<b$, and let $g:I\to\R$ be nondecreasing and left-continuous.  We normalize the associated finite Lebesgue--Stieltjes measure by
\begin{equation}\label{eq:measure}
 \mug([s,t))=g(t)-g(s),
 \qquad a\le s<t\le b.
\end{equation}
The set of right jumps is
\[
 \Dg:=\{\tau\in[a,b):\Delta^+g(\tau):=g(\tau+)-g(\tau)>0\}.
\]
At $\tau\in\Dg$ the atomic mass is
\[
 m_\tau:=\mug(\{\tau\})=\Delta^+g(\tau).
\]
We write
\begin{equation}\label{eq:measure-decomp}
 \mug=\mugc+\sum_{\tau\in\Dg}m_\tau\delta_\tau,
\end{equation}
where $\mugc$ is nonatomic; it may itself contain absolutely continuous and singular-continuous parts.

We work with the canonical $g$-absolutely continuous class.  A function $x:I\to\R^n$ belongs to $\ACg(I;\R^n)$ if there exists $v\in L^1(\mug;\R^n)$ such that
\begin{equation}\label{eq:AC-representation}
 x(t)=x(a)+\int_{[a,t)}v(s)\,d\mug(s).
\end{equation}
The density $v$, unique $\mug$-almost everywhere, is denoted by $x'_g$.  The fundamental theorem for the Lebesgue--Stieltjes integral is given in \cite[Theorem~5.1]{FrigonPouso2017}.  In particular, at every atom,
\begin{equation}\label{eq:jump}
 x(\tau+)-x(\tau)=m_\tau x'_g(\tau).
\end{equation}
The canonical representative is important because the pointwise Stieltjes derivative can have a nontrivial kernel in larger function classes; see \citet{FernandezKernel2025}.

For $\alpha>0$, define
\begin{equation}\label{eq:holder}
 [x]_{\alpha,g}
 :=\sup_{\substack{s,t\in I\\g(s)\ne g(t)}}
 \frac{\norm{x(t)-x(s)}}{\abs{g(t)-g(s)}^\alpha},
\end{equation}
with the additional convention that $x(t)=x(s)$ whenever $g(t)=g(s)$.  We refer to $\alpha=1$ as the $g$-Lipschitz case and to $\alpha>1$ as the super-$g$-H\"older regime, following \citet{IlterDuruKoca2026}.

\subsection{Controlled dynamics and the jump law}

The control system is written with respect to this generalized clock.  Away from atoms it behaves as a measure-driven differential equation; at an atom it becomes a finite state update.  The latter viewpoint will be decisive later.

Consider
\begin{equation}\label{eq:general-system}
 x'_g(t)=A(t)x(t)+B(t)u(t),
 \qquad x(a)=x_0,
\end{equation}
where
\[
 A\in L^\infty(\mug;\R^{n\times n}),
 \qquad
 B\in L^\infty(\mug;\R^{n\times m}),
 \qquad
 u\in\cU_g:=L^2(\mug;\R^m).
\]
We fix Borel representatives whenever pointwise atomic values are used.  This convention is harmless at atoms: if two representatives agree $\mug$-almost everywhere, then they agree on every singleton of positive $\mug$-measure.  Since $\mug(I)<\infty$, $L^2(\mug)\subset L^1(\mug)$, and the state equation is meaningful in the canonical class.

At an atom $\tau$ the jump law is
\begin{equation}\label{eq:controlled-jump}
 x(\tau+)-x(\tau)
 =m_\tau\bigl(A(\tau)x(\tau)+B(\tau)u(\tau)\bigr).
\end{equation}

\subsection{Ordinary finite-atomic reachability: what it sees and what it does not}

Before imposing any regularity constraint, the finite-atomic problem is an ordinary reachability problem for a discrete sequence of event updates.  The result needed later is already established in \citet{FernandezReachability2026}.  We recall it not as a new theorem, but because it provides the baseline against which the singular regularity effect will be measured.

The key point is that ordinary reachability records which directions can be generated after a given number of events; it does not record how small a positive event mass is.  This insensitivity is exactly what makes the limit $m_\ell\downarrow0$ nontrivial once a trajectory-regularity budget is imposed.

For constant $A\in\R^{n\times n}$ and $B\in\R^{n\times m}$ define
\begin{equation}\label{eq:Kdef}
 \cK_r(A,B)
 :=\Span\{\Ran B,\Ran AB,\ldots,\Ran A^{r-1}B\},
 \qquad r\ge1,
\end{equation}
and $\cK_0(A,B):=\{0\}$.  The controllability index is
\[
 \nu(A,B):=\min\{r\ge1:\cK_r(A,B)=\cK_n(A,B)\}.
\]

\begin{theorem}[Finite-atomic ordinary reachability, \citealp{FernandezReachability2026}]\label{thm:sourceA-atomic}
Assume that $A,B$ are constant and
\begin{equation}\label{eq:pure-atomic-clock}
 \mug=\sum_{j=1}^Nm_j\delta_{\tau_j},
 \qquad
 a\le\tau_1<\cdots<\tau_N<b,
 \qquad m_j>0.
\end{equation}
Then, from the zero initial state,
\begin{equation}\label{eq:ordinary-KN}
 \cR_g(a,b;0)=\cK_N(A,B).
\end{equation}
Consequently, exact controllability holds if and only if
\[
 \operatorname{rank}[B,AB,\ldots,A^{N-1}B]=n.
\]
If $(A,B)$ is Kalman controllable, the least number of positive atoms required for exact controllability is $\nu(A,B)$.  The reachable subspace is independent of the locations and positive values of the masses.
\end{theorem}

The theorem says that, from the viewpoint of ordinary reachability, every positive event counts as one opportunity to enlarge the Krylov span.  Its precise positive mass changes the numerical state transition and control cost, but not the reachable subspace.  Later we shall see that a uniform super-$g$-H\"older budget breaks this insensitivity: an event can continue to count for ordinary reachability while becoming too small to carry a bounded normalized jump.

For later use, the underlying forward recurrence is
\begin{equation}\label{eq:atomic-rec-background}
 x_j=(I+m_jA)x_{j-1}+m_jBu_j,
 \qquad j=1,\ldots,N,
\end{equation}
where $x_j$ denotes the state just after the $j$th event and $x_0$ the state before the first event.  No inverse jump matrix is needed in this recurrence, so the finite-atomic identity does not require a regressivity hypothesis.

\section{The regularity filter}\label{sec:regularity}

The singular limit studied later depends on a sharp change at $\alpha=1$.  At the critical exponent, the seminorm controls the Stieltjes derivative in the usual bounded-velocity sense.  Above it, nonatomic motion disappears and only atomic jumps remain.  We record exactly these two facts.

\begin{proposition}[The critical identity]\label{prop:alpha1}
If $x\in AC_g(I;\mathbb R^n)$, then, with extended values,
\begin{equation}\label{eq:alpha1}
 [x]_{1,g}=\|x'_g\|_{L^\infty(\mu_g)}.
\end{equation}
Hence a controlled trajectory satisfies
\[
 [x]_{1,g}=\|A(\cdot)x(\cdot)+B(\cdot)u(\cdot)\|_{L^\infty(\mu_g)}.
\]
\end{proposition}

\begin{proof}
The integral representation gives $[x]_{1,g}\le\|x'_g\|_\infty$.  Conversely, at non-atomic points the canonical derivative is the limit of $g$-difference quotients, while at an atom $\tau$ one has $x(\tau+)-x(\tau)=m_\tau x'_g(\tau)$.  Both are bounded by $[x]_{1,g}$ outside a common $\mu_g$-null set.  This yields the reverse inequality.
\end{proof}

For $\alpha>1$ the picture changes.  The following is the controlled specialization of the pure-jump representation of \citet{IlterDuruKoca2026}.

\begin{proposition}[Super-$g$-H\"older regularity filter]\label{prop:filter}
Let $\alpha>1$ and let $(x,u)$ solve \eqref{eq:general-system}.  If $[x]_{\alpha,g}<\infty$, then
\begin{equation}\label{eq:nonatomic-cancel}
 A(t)x(t)+B(t)u(t)=0\qquad \mu_g^{\rm c}\text{-a.e.},
\end{equation}
and
\begin{equation}\label{eq:super-seminorm}
 [x]_{\alpha,g}
 =\sup_{\tau\in D_g}
 \frac{\|A(\tau)x(\tau)+B(\tau)u(\tau)\|}{m_\tau^{\alpha-1}}.
\end{equation}
The trajectory is therefore determined by its atomic jumps.
\end{proposition}

\begin{proof}
By \citet[Proposition~3.2]{IlterDuruKoca2026}, the canonical derivative of a super-$g$-H\"older map vanishes on the nonatomic part of the Stieltjes measure.  The state equation gives \eqref{eq:nonatomic-cancel}.  At each atom, the jump law \eqref{eq:controlled-jump} and \citet[Theorem~3.3]{IlterDuruKoca2026} give the exact normalized-jump formula \eqref{eq:super-seminorm} and the corresponding pure-jump representation.
\end{proof}

The contrast is the one needed later.  At $\alpha=1$, a trajectory may still move along a continuous part of the clock while keeping bounded Stieltjes velocity.  For $\alpha>1$, nonatomic state variation is suppressed; if the clock is continuous, the trajectory must be constant.  Nontrivial supercritical motion therefore survives only through atoms, so from now on we restrict attention to finite atomic clocks and ask what happens when one of those atoms disappears.
\section{Finite atomic clocks and the regularity cost}\label{sec:atomic}

Assume from now on that the clock is purely atomic as in \eqref{eq:pure-atomic-clock}, that $A,B$ are constant, that $x_0=0$, and that $\alpha>1$.  The natural variables for the singular problem are not the controls themselves, but the state jumps measured on the scale imposed by regularity.  If $x_j$ is the state immediately after event $j$, define the normalized jumps
\begin{equation}\label{eq:wdef}
 w_j:=\frac{x_j-x_{j-1}}{m_j^\alpha},\qquad j=1,\ldots,N,
\end{equation}
These variables satisfy a simple event equation.  The recurrence \eqref{eq:atomic-rec-background} then gives
\begin{equation}\label{eq:w-dynamics}
 m_j^{\alpha-1}w_j=Ax_{j-1}+Bu_j.
\end{equation}
Summing the jumps also gives
\begin{equation}\label{eq:x-from-w}
 x_j=\sum_{k=1}^j m_k^\alpha w_k.
\end{equation}
By \citet[Theorem~3.3]{IlterDuruKoca2026},
\begin{equation}\label{eq:atomic-holder-exact}
 [x]_{\alpha,g}=\max_{1\le j\le N}\|w_j\|.
\end{equation}
Thus the super-$g$-H\"older cost has a simple eventwise meaning: it is the largest state jump after each jump has been divided by the scale assigned to the event that carries it.

We now express feasibility directly in these variables.  Set $X_N=(\mathbb R^n)^N$ with $\|w\|_{\infty,2}=\max_j\|w_j\|$, and define
\begin{equation}\label{eq:V}
 \mathcal V_{\alpha,\mathbf m}(A,B)
 :=\left\{w\in X_N:
 m_j^{\alpha-1}w_j-A\sum_{k<j}m_k^\alpha w_k\in\operatorname{Ran}B,
 \ j=1,\ldots,N\right\}.
\end{equation}
Its endpoint map is
\begin{equation}\label{eq:T}
 T_{\alpha,\mathbf m}w:=\sum_{j=1}^Nm_j^\alpha w_j.
\end{equation}
For a terminal direction $d\in\mathbb R^n$ define
\begin{equation}\label{eq:Lambda-def}
 \Lambda_{\alpha,\mathbf m}(d)
 :=\inf\{[x]_{\alpha,g}:x_0=0,\ x_N=d,\ x\text{ satisfies \eqref{eq:atomic-rec-background}}\},
\end{equation}
with value $+\infty$ when $d$ is not ordinarily reachable.

\begin{proposition}[Finite-dimensional minimum-amplitude formulation]\label{prop:quotient}
The elements of $\mathcal V_{\alpha,\mathbf m}(A,B)$ are exactly the normalized-jump sequences of controlled trajectories, and
\[
 T_{\alpha,\mathbf m}(\mathcal V_{\alpha,\mathbf m}(A,B))=\mathcal K_N(A,B).
\]
For $d\in\mathcal K_N(A,B)$,
\begin{equation}\label{eq:Lambda-w}
 \Lambda_{\alpha,\mathbf m}(d)
 =\min_{\substack{w\in\mathcal V_{\alpha,\mathbf m}(A,B)\\T_{\alpha,\mathbf m}w=d}}
 \|w\|_{\infty,2},
\end{equation}
and the minimum is attained.  Moreover,
\begin{equation}\label{eq:basic-lower}
 \Lambda_{\alpha,\mathbf m}(d)
 \ge \frac{\|d\|}{\sum_{j=1}^Nm_j^\alpha}.
\end{equation}
\end{proposition}

\begin{proof}
Equations \eqref{eq:w-dynamics} show that a controlled trajectory yields precisely the constraints in \eqref{eq:V}; conversely those constraints allow a control $u_j$ with $Bu_j=m_j^{\alpha-1}w_j-Ax_{j-1}$.  The endpoint is \eqref{eq:T}, whose range is the ordinary reachable space $\mathcal K_N(A,B)$ by \cref{thm:sourceA-atomic}.  Formula \eqref{eq:atomic-holder-exact} gives \eqref{eq:Lambda-w}.  For fixed reachable $d$, the feasible set
\[
 F_d:=\{w\in\mathcal V_{\alpha,\mathbf m}(A,B):T_{\alpha,\mathbf m}w=d\}
\]
is a nonempty closed affine subset of the finite-dimensional space $X_N$.  Let $\lambda$ be the infimum in \eqref{eq:Lambda-w} and choose a minimizing sequence $(w^{(r)})\subset F_d$ with $\|w^{(r)}\|_{\infty,2}\le\lambda+1$ for all sufficiently large $r$.  That tail lies in the compact set $F_d\cap\{\|w\|_{\infty,2}\le\lambda+1\}$, so a convergent subsequence has a limit in $F_d$; continuity of the norm shows that the limit attains $\lambda$.  Finally,
\[
 \|d\|\le\sum_{j=1}^Nm_j^\alpha\|w_j\|
 \le\left(\sum_{j=1}^Nm_j^\alpha\right)\|w\|_{\infty,2},
\]
which yields \eqref{eq:basic-lower}.
\end{proof}

Once written in normalized-jump variables, \eqref{eq:Lambda-w} is a classical finite-dimensional minimum-amplitude problem of the type studied by \citet{Cadzow1974,Kolev1975}.  What matters here is not that optimization principle.  The issue is what happens to its feasible geometry when one event weight tends to zero.
\section{Vanishing events: exact bounded-cost reachability}\label{sec:vanishing}

We now let one event disappear.  Ordinary reachability cannot see this process until the event weight actually reaches zero: every positive event still contributes to the same truncated Krylov geometry.  A uniform regularity budget reacts earlier, because the jump at the vulnerable event is measured against a shrinking scale.  It is useful to think of the trajectory in three pieces.  The controls before event $\ell$ build an intermediate state; the vulnerable event tests whether that state is compatible with the singular limit; the remaining controls then complete the terminal motion.

Fix an event index $\ell\in\{1,\ldots,N\}$.  All masses except the $\ell$th are fixed and positive, while
\begin{equation}\label{eq:epsilon-mass}
 m_\ell=\eps>0,
 \qquad \eps\downarrow0.
\end{equation}
Write $\mathbf m(\eps)$ for the resulting mass vector and abbreviate
\[
 \Lam_\eps(d):=\Lam_{\alpha,\mathbf m(\eps)}(d).
\]
Let $Q$ be the Euclidean orthogonal projection onto $(\Ran B)^\perp$ and define
\begin{equation}\label{eq:Sdef}
 S:=\Ker(QA)=A^{-1}(\Ran B)
 =\{y\in\R^n:Ay\in\Ran B\}.
\end{equation}
The subspace $S$ has a direct control interpretation: at a state $y\in S$, the drift $Ay$ can be cancelled by some control value because there exists $u$ with $Ay+Bu=0$.  At a state outside $S$, a component of the drift points in a direction that the control cannot instantaneously oppose.

The states reachable immediately before the disappearing event form
\begin{equation}\label{eq:Rpre}
 \Rpre:=\cK_{\ell-1}(A,B).
\end{equation}
Let
\begin{equation}\label{eq:Pplus}
 P_+:=\prod_{j=N}^{\ell+1}(I+m_jA),
\end{equation}
where the empty product is the identity.  This is the homogeneous propagation from immediately after event $\ell$ to the terminal time.  The terminal subspace generated by controls acting strictly after event $\ell$ is
\begin{equation}\label{eq:Rpost}
 \Rpost:=\cK_{N-\ell}(A,B).
\end{equation}
If event $\ell$ is deleted, the ordinary reachable space is
\begin{equation}\label{eq:Rdeleted}
 \Rdel=\cK_{N-1}(A,B)=P_+\Rpre+\Rpost.
\end{equation}

At this point ordinary reachability and the singular regularity problem can already be distinguished.  The space $\Rdel$ asks only whether the remaining $N-1$ events can reach a target after event $\ell$ is removed.  The set defined next asks a stronger question: can the same target be reached along the original $N$-event systems with a regularity cost that stays bounded uniformly as the $\ell$th event disappears?

\begin{definition}[Uniformly bounded regularity-conditioned reachable subspace]\label{def:B}
Define
\begin{equation}\label{eq:Bdef}
 \cB_\ell
 :=\left\{d\in\R^n:
 \sup_{0<\eps<\eps_0}\Lam_\eps(d)<\infty
 \text{ for some }\eps_0>0\right\}.
\end{equation}
\end{definition}

\subsection{Why a vanishing event leaves a compatibility condition}

The decisive calculation is local to the disappearing event.  It explains the appearance of $A^{-1}(\Ran B)$ before any terminal-space algebra is used.  At the disappearing event, \eqref{eq:w-dynamics} reads
\begin{equation}\label{eq:epsilon-stage}
 \eps^{\alpha-1}w_\ell=Ax_{\ell-1}+Bu_\ell.
\end{equation}
Applying $Q$ yields the fundamental identity
\begin{equation}\label{eq:hidden-Q}
 QA x_{\ell-1}=\eps^{\alpha-1}Qw_\ell.
\end{equation}
The point of projecting with $Q$ is that the control term disappears.  Because $\alpha>1$, the prefactor $\eps^{\alpha-1}$ tends to zero, so bounded normalized jumps force the component of $Ax_{\ell-1}$ outside $\Ran B$ to vanish.  This is the source of the hidden compatibility condition; the rest of the theorem propagates that local restriction to the terminal space.

\begin{theorem}[Exact bounded-cost subspace, propagated form]\label{thm:B-propagated}
For every $\alpha>1$,
\begin{equation}\label{eq:B-propagated}
 \cB_\ell=P_+(\Rpre\cap S)+\Rpost.
\end{equation}
Equivalently, if $d$ does not belong to the right-hand side, then
\begin{equation}\label{eq:diverge-general}
 \Lam_\eps(d)\longrightarrow+\infty
 \qquad(\eps\downarrow0).
\end{equation}
No Kalman-controllability assumption is required.
\end{theorem}

The formula mirrors the three stages just described.  The limiting pre-event state must lie in $\Rpre$, because only the first $\ell-1$ events can create it.  It must also lie in $S$, because its drift has to be cancellable when the vulnerable event loses its weight.  Once that compatible state has passed the singular stage, the later controls may add any direction in $\Rpost$.

\begin{proof}
The two inclusions use different ideas.  For necessity, start from an arbitrary family with uniformly bounded cost and pass to a limiting pre-event state; the projection identity then forces compatibility.  For sufficiency, start from a compatible pre-event state and cancel its drift exactly at the vulnerable event.

\emph{Necessity: bounded cost forces compatibility.}  Let $d\in\cB_\ell$.  By definition there are $C>0$ and $\eps_0>0$ such that $\Lam_\eps(d)\le C$ for $0<\eps<\eps_0$.  Choose an arbitrary sequence $\eps_r\downarrow0$ and, for each $r$, a minimizing normalized-jump sequence $w^{(r)}$ as in \cref{prop:quotient}.  Then
\begin{equation}\label{eq:w-bound}
 \norm{w_j^{(r)}}\le C
 \qquad(j=1,\ldots,N).
\end{equation}
For $j<\ell$ all masses are fixed and positive, so \eqref{eq:x-from-w} gives a uniform bound on $x_j^{(r)}$.  In particular, $(x_{\ell-1}^{(r)})_r$ is bounded.  Passing to a subsequence,
\begin{equation}\label{eq:pre-limit}
 x_{\ell-1}^{(r)}\longrightarrow y
 \qquad\text{for some }y\in\R^n.
\end{equation}
Every $x_{\ell-1}^{(r)}$ belongs to the fixed finite-dimensional subspace $\Rpre$, which is closed; hence $y\in\Rpre$.

Projecting \eqref{eq:epsilon-stage} and using \eqref{eq:w-bound},
\[
 \norm{QA x_{\ell-1}^{(r)}}
 =\eps_r^{\alpha-1}\norm{Qw_\ell^{(r)}}
 \le C\eps_r^{\alpha-1}\longrightarrow0.
\]
Continuity of $QA$ and \eqref{eq:pre-limit} give $QAy=0$, so $y\in S$.

The state jump at the disappearing event satisfies
\begin{equation}\label{eq:jump-goes-zero}
 \norm{x_\ell^{(r)}-x_{\ell-1}^{(r)}}
 =\eps_r^\alpha\norm{w_\ell^{(r)}}
 \le C\eps_r^\alpha\longrightarrow0,
\end{equation}
therefore $x_\ell^{(r)}\to y$.  For each $r$, split the terminal state into the homogeneous propagation of $x_\ell^{(r)}$ and the contribution of the post-event controls:
\begin{equation}\label{eq:terminal-split-r}
 d=P_+x_\ell^{(r)}+r_+^{(r)},
 \qquad r_+^{(r)}\in\Rpost.
\end{equation}
Since $x_\ell^{(r)}\to y$ and $d$ is fixed, $r_+^{(r)}\to d-P_+y$.  The subspace $\Rpost$ is closed, hence $d-P_+y\in\Rpost$.  We have found
\[
 y\in\Rpre\cap S,
 \qquad
 d=P_+y+r_+,
 \quad r_+\in\Rpost,
\]
which proves
\[
 \cB_\ell\subseteq P_+(\Rpre\cap S)+\Rpost.
\]

\emph{Sufficiency: a compatible intermediate state gives a uniform recovery family.}  Let
\[
 d=P_+y+r_+,
 \qquad
 y\in\Rpre\cap S,
 \quad r_+\in\Rpost.
\]
Because $y\in\Rpre$, choose once and for all controls at the first $\ell-1$ events that steer zero to $y$.  Their masses do not depend on $\eps$, so their normalized jumps and their contribution to the maximum norm are fixed.

Since $y\in S$, $Ay\in\Ran B$.  Choose a vector $u_\ell^*$ with
\begin{equation}\label{eq:cancel-control}
 Bu_\ell^*=-Ay.
\end{equation}
At the $\ell$th event, \eqref{eq:atomic-rec-background} then gives
\[
 x_\ell=(I+\eps A)y+\eps Bu_\ell^*=y.
\]
Thus the disappearing event produces exactly zero state jump and
\begin{equation}\label{eq:zero-well}
 w_\ell=0
\end{equation}
for every $\eps>0$.

Finally, because $r_+\in\Rpost$, choose fixed post-event controls whose terminal control contribution is $r_+$.  Starting from $y$ after event $\ell$, the terminal state is then $P_+y+r_+=d$.  All post-event masses are fixed, and the corresponding normalized state jumps depend only on the fixed state $y$, fixed masses and fixed controls; they are independent of $\eps$.  The maximum in \eqref{eq:atomic-holder-exact} is therefore bounded by a constant independent of $\eps$.  Hence $d\in\cB_\ell$, proving the reverse inclusion.

It remains to justify the divergence statement.  Suppose $d\notin\cB_\ell$ but $\Lam_\eps(d)$ did not tend to $+\infty$.  Then there would exist $\eps_r\downarrow0$ and $C<\infty$ with $\Lam_{\eps_r}(d)\le C$.  Repeating the necessity argument along this sequence would put $d$ in the right-hand side of \eqref{eq:B-propagated}; the sufficiency argument would then imply $d\in\cB_\ell$, a contradiction.
\end{proof}

\subsection{Mass-independent Krylov form}

The propagated form is natural for the proof, but it still contains the mass-dependent matrix $P_+$.  The next observation removes it.  Once the pre-event state is compatible, homogeneous propagation after the disappearing event changes that state only by directions that the post-event controls can already generate.  Consequently the bounded-cost \emph{subspace} is independent of the remaining positive masses, even though the numerical value of the cost is not.

\begin{lemma}[Post-event propagation of a compatible state]\label{lem:prop-simplify}
Let $q=N-\ell$.  If $y\in S$, then
\begin{equation}\label{eq:prop-mod}
 P_+y-y\in\cK_q(A,B)=\Rpost.
\end{equation}
\end{lemma}

\begin{proof}
Since $P_+$ is a product of $q$ factors of the form $I+m_jA$, there are scalars $c_1,\ldots,c_q$ such that
\[
 P_+=I+c_1A+\cdots+c_qA^q.
\]
Because $y\in S$, there exists $v\in\R^m$ with $Ay=Bv$.  Consequently,
\[
 A^ry=A^{r-1}Bv\in\Ran(A^{r-1}B)\subseteq\cK_q(A,B)
 \qquad(1\le r\le q).
\]
Therefore
\[
 P_+y-y=\sum_{r=1}^q c_rA^ry\in\cK_q(A,B).
\]
For $q=0$ the statement is the empty-product identity $P_+=I$.
\end{proof}

\begin{theorem}[Exact mass-independent Krylov characterization]\label{thm:main}
For every $\alpha>1$,
\begin{equation}\label{eq:main}
 \cB_\ell
 =\bigl(\cK_{\ell-1}(A,B)\cap A^{-1}(\Ran B)\bigr)
  +\cK_{N-\ell}(A,B).
\end{equation}
Thus $\cB_\ell$ is a linear subspace.  It depends on $(A,B)$ and on the position $\ell$ of the disappearing event, but not on the positive values of the remaining masses.  It is the same subspace for every exponent $\alpha>1$.
\end{theorem}

Formula \eqref{eq:main} is the main structural result.  Its first term is the set of pre-event states that can actually be built and that remain compatible with the vanishing event.  Its second term contains the corrections still available afterwards.  The event position $\ell$ therefore matters, whereas the positive values of the remaining event weights disappear from the limiting subspace.  They can still affect the numerical value of the minimum cost.

\begin{proof}
By \cref{thm:B-propagated},
\[
 \cB_\ell=P_+(\Rpre\cap S)+\Rpost.
\]
For $y\in\Rpre\cap S$, \cref{lem:prop-simplify} gives $P_+y=y+r$ for some $r\in\Rpost$.  Hence
\[
 P_+(\Rpre\cap S)+\Rpost
 =(\Rpre\cap S)+\Rpost.
\]
Substitute \eqref{eq:Rpre}, \eqref{eq:Rpost} and \eqref{eq:Sdef}.  The right-hand side of \eqref{eq:main} is a sum of linear subspaces, hence itself a linear subspace.  The displayed formula contains neither the remaining mass values nor $\alpha$; only the assumption $\alpha>1$ is used in the necessity argument through $\eps^{\alpha-1}\to0$.
\end{proof}

\begin{remark}[Why $\alpha=1$ is different]\label{rem:alpha1-C2}
At $\alpha=1$, identity \eqref{eq:hidden-Q} becomes $QA x_{\ell-1}=Qw_\ell$ and no small factor forces $QA x_{\ell-1}$ to vanish.  The compatibility mechanism of \cref{thm:main} is therefore genuinely supercritical.  This is consistent with \cref{prop:alpha1}: $g$-Lipschitz regularity bounds the derivative but does not suppress it.
\end{remark}

\subsection{Comparison with the deleted-event reachable space}

The spaces $\Rdel$ and $\mathcal B_\ell$ answer different questions.  The first asks what remains reachable after event $\ell$ has been removed.  The second asks which of those targets can be approached through the original systems without the regularity cost blowing up.  A bounded-cost family has a deleted-event limit, so $\mathcal B_\ell\subseteq\Rdel$.  The inclusion may be strict because ordinary deleted-event reachability places no restriction on the intermediate state at the missing event.

\begin{proposition}[Ordinary versus uniformly bounded regularity-conditioned reachability]\label{prop:B-v-Rdel}
One always has
\begin{equation}\label{eq:B-sub-Rdel}
 \cB_\ell\subseteq\Rdel=\cK_{N-1}(A,B).
\end{equation}
If
\begin{equation}\label{eq:auto-compat}
 A\cK_{\ell-1}(A,B)\subseteq\Ran B,
\end{equation}
then the hidden compatibility restriction is automatic and
\begin{equation}\label{eq:B-eq-Rdel}
 \cB_\ell=\Rdel.
\end{equation}
In particular, equality holds if $B$ is surjective.  It also holds for $A=0$.
\end{proposition}

\begin{proof}
The inclusion follows immediately from the propagated form \eqref{eq:B-propagated}, because $P_+\Rpre+\Rpost=\Rdel$.  If \eqref{eq:auto-compat} holds, then $\Rpre\subseteq S$, so
\[
 \cB_\ell=P_+\Rpre+\Rpost=\Rdel.
\]
If $B$ is surjective, $\Ran B=\R^n$ and therefore $S=\R^n$.  If $A=0$, again $S=\R^n$.
\end{proof}

\begin{corollary}[First, last and sufficiently early disappearing events]\label{cor:event-position}
The following identities hold:
\begin{align}
 \cB_1&=\cK_{N-1}(A,B)=\Rdel,\label{eq:Bfirst}\\
 \cB_N&=\cK_{N-1}(A,B)\cap A^{-1}(\Ran B).\label{eq:Blast}
\end{align}
More generally, if $N-\ell\ge\ell-1$, then
\begin{equation}\label{eq:early}
 \cB_\ell=\cK_{N-\ell}(A,B).
\end{equation}
\end{corollary}

\begin{proof}
For $\ell=1$, \eqref{eq:main} gives $\cB_1=(\cK_0\cap S)+\cK_{N-1}=\cK_{N-1}$.  For $\ell=N$, the post-event space is $\cK_0=\{0\}$.  Finally, if $N-\ell\ge\ell-1$, then $\cK_{\ell-1}\subseteq\cK_{N-\ell}$, so the first summand in \eqref{eq:main} is already contained in the second.
\end{proof}

\section{Limit cost and two degeneration mechanisms}\label{sec:costlimit}

The subspace theorem tells us when the cost stays bounded, but not what the bounded cost converges to or why it diverges outside that subspace.  There are two geometrically different failures.  A target may disappear from the ordinary reachable set when the event is removed.  Or the target may remain ordinarily reachable, while every deleted-event realization passes through an incompatible pre-event state.  These mechanisms lead to different blow-up scales and should be treated separately.

\subsection{The compatible deleted-event limit cost}

Suppose first that $d\in\cB_\ell$.  The preceding theorem suggests the correct limiting problem: delete event $\ell$, but keep the compatibility condition that survived the singular limit.  Thus, among deleted-event trajectories reaching $d$, we retain only those whose state $y$ immediately before the missing event lies in $S$.  Define
\begin{equation}\label{eq:Lambda0}
 \Lam_0^{(\ell)}(d)
 :=\min\max_{j\ne\ell}
 \frac{\norm{x_j-x_{j-1}}}{m_j^\alpha},
\end{equation}
where the minimum is taken over those compatible deleted-event trajectories.  The feasible set is nonempty exactly for $d\in\cB_\ell$, and finite-dimensional compactness after restriction to a sufficiently large norm ball gives existence of a minimizer.

\begin{theorem}[Convergence to the compatible deleted-event cost]\label{thm:cost-convergence}
For every $d\in\cB_\ell$,
\begin{equation}\label{eq:cost-convergence}
 \Lam_\eps(d)\longrightarrow\Lam_0^{(\ell)}(d)
 \qquad(\eps\downarrow0).
\end{equation}
Moreover,
\begin{equation}\label{eq:cost-upper-all}
 \Lam_\eps(d)\le\Lam_0^{(\ell)}(d)
 \qquad\text{for every }\eps>0.
\end{equation}
\end{theorem}

\begin{proof}
Take a minimizing compatible deleted-event trajectory for $\Lam_0^{(\ell)}(d)$.  At the missing event, its pre-event state $y$ satisfies $Ay\in\Ran B$.  Insert the event with mass $\eps$ and choose the cancellation control $Bu_\ell=-Ay$.  The inserted event then has zero state jump, and every other state and normalized jump is unchanged.  This produces a full $N$-event trajectory of exactly the same cost, proving \eqref{eq:cost-upper-all} and the limsup inequality
\[
 \limsup_{\eps\downarrow0}\Lam_\eps(d)
 \le\Lam_0^{(\ell)}(d).
\]

For the reverse inequality, choose a sequence $\eps_r\downarrow0$ realizing the liminf and minimizers for $\Lam_{\eps_r}(d)$.  The upper bound just proved makes their costs uniformly bounded.  Repeat the compactness argument in the proof of \cref{thm:B-propagated}.  After passing to a subsequence, the pre-event states converge to a compatible state $y\in\Rpre\cap S$, the disappearing state jump tends to zero, and the fixed-mass normalized jumps admit convergent subsequences.  The resulting deleted-event trajectory reaches $d$ and is compatible.  The maximum of finitely many Euclidean norms is lower semicontinuous, so its cost is no larger than
\[
 \liminf_{r\to\infty}\Lam_{\eps_r}(d).
\]
By minimality of $\Lam_0^{(\ell)}(d)$,
\[
 \Lam_0^{(\ell)}(d)
 \le\liminf_{\eps\downarrow0}\Lam_\eps(d).
\]
Together with the limsup inequality this proves \eqref{eq:cost-convergence}.
\end{proof}

\subsection{Range loss: the target itself disappears from the reachable set}

Range loss is the simpler mechanism.  If the target is not reachable after deletion, the nonvanishing events leave a genuine terminal component unresolved.  In the $\eps$-system that missing component can only be supplied through the vulnerable state jump.  Because this jump is $\eps^\alpha w_\ell$, an $O(1)$ missing component forces a normalized burden on the scale $\eps^{-\alpha}$.

\begin{proposition}[Range-loss lower bound]\label{prop:range-loss}
Let $d\notin\Rdel$.  If $P_+\ne0$, then
\begin{equation}\label{eq:range-lower}
 \Lam_\eps(d)
 \ge
 \frac{\dist(d,\Rdel)}{\norm{P_+}\eps^\alpha}.
\end{equation}
If $P_+=0$, then every terminal state reachable by the full system belongs to $\Rdel$, so $\Lam_\eps(d)=+\infty$ for $d\notin\Rdel$.
\end{proposition}

\begin{proof}
Take any full trajectory reaching $d$ and let $w_\ell$ be its normalized jump at the disappearing event.  Construct a comparison terminal state $d_0$ by keeping the same pre-event history and the same post-event controls but deleting the $\ell$th state jump.  The comparison system has only $N-1$ effective events, so $d_0\in\Rdel$.  By linearity, the effect of the deleted state jump on the terminal state is its homogeneous propagation:
\begin{equation}\label{eq:difference-range}
 d-d_0=\eps^\alpha P_+w_\ell.
\end{equation}
Thus, when $P_+\ne0$,
\[
 \dist(d,\Rdel)
 \le\norm{d-d_0}
 \le\eps^\alpha\norm{P_+}\norm{w_\ell}
 \le\eps^\alpha\norm{P_+}[x]_{\alpha,g}.
\]
Taking the infimum over full trajectories gives \eqref{eq:range-lower}.  If $P_+=0$, \eqref{eq:difference-range} gives $d=d_0$ for every full trajectory, so no target outside $\Rdel$ is reachable.
\end{proof}

\begin{example}[Sharp range-loss scale]\label{ex:scalar-range}
Let $n=m=N=1$, $A=0$, $B=1$, and let the unique mass be $\eps$.  For target $d\ne0$,
\[
 d=\eps^\alpha w_1,
 \qquad
 \Lam_\eps(d)=\frac{|d|}{\eps^\alpha}.
\]
The deleted-event reachable space is $\{0\}$, so the exponent $\eps^{-\alpha}$ in \cref{prop:range-loss} is sharp in general.
\end{example}

\subsection{Compatibility loss: ordinary reachability survives}

Compatibility loss is subtler.  Now $d\in\Rdel$, so the target survives deletion.  The cost may still diverge if every deleted-event trajectory reaching $d$ passes through a pre-event state whose drift has a component outside $\Ran B$.  The next quantity measures the smallest possible size of that incompatible drift component.  It is therefore a quantitative compatibility defect, not just an auxiliary minimization.

Define
\begin{equation}\label{eq:rho-def}
 \rho_\ell(d)
 :=\min\left\{\norm{QAy}:
 y\in\Rpre,\ r\in\Rpost,\ P_+y+r=d\right\}.
\end{equation}
The minimum is attained: the set of possible values $QA y$ is a closed affine subspace of the finite-dimensional space $QA\Rpre$, so its distance from the origin is realized.

\begin{lemma}[Compatibility seminorm]\label{lem:rho}
The map $\rho_\ell$ is a continuous seminorm on $\Rdel$ and
\begin{equation}\label{eq:rho-kernel}
 \Ker\rho_\ell=\cB_\ell.
\end{equation}
Consequently there exists $c_\ell>0$ such that
\begin{equation}\label{eq:rho-cont}
 \rho_\ell(z)\le c_\ell\norm z
 \qquad(z\in\Rdel).
\end{equation}
\end{lemma}

\begin{proof}
Positive homogeneity follows by scaling a feasible decomposition.  For $d_1,d_2\in\Rdel$, choose minimizing or arbitrarily near-minimizing feasible pre-states $y_1,y_2$.  Then $y_1+y_2$ is feasible for $d_1+d_2$, so the triangle inequality for the Euclidean norm gives
\[
 \rho_\ell(d_1+d_2)\le\rho_\ell(d_1)+\rho_\ell(d_2).
\]
Thus $\rho_\ell$ is a seminorm.  Its kernel consists exactly of targets admitting a decomposition $d=P_+y+r$ with $y\in\Rpre$ and $QAy=0$, which is the propagated characterization \eqref{eq:B-propagated}.  Every seminorm on a finite-dimensional normed space is continuous, yielding \eqref{eq:rho-cont}.
\end{proof}

\begin{proposition}[Compatibility-loss lower bound]\label{prop:compat-loss}
For every $d\in\Rdel$,
\begin{equation}\label{eq:compat-lower}
 \Lam_\eps(d)
 \ge
 \frac{\rho_\ell(d)}
 {\eps^{\alpha-1}\bigl(1+c_\ell\norm{P_+}\eps\bigr)}.
\end{equation}
Hence, if
\[
 d\in\Rdel\setminus\cB_\ell,
\]
then
\begin{equation}\label{eq:compat-omega}
 \Lam_\eps(d)=\Omega(\eps^{1-\alpha}).
\end{equation}
\end{proposition}

\begin{proof}
Take any full trajectory reaching $d$ and put $L=[x]_{\alpha,g}$.  As in \cref{prop:range-loss}, delete only the $\ell$th state jump and denote the resulting terminal state by
\[
 d_0=d-\eps^\alpha P_+w_\ell\in\Rdel.
\]
The unchanged pre-event state $x_{\ell-1}$ is feasible in the minimization defining $\rho_\ell(d_0)$, hence by \eqref{eq:hidden-Q},
\[
 \rho_\ell(d_0)
 \le\norm{QA x_{\ell-1}}
 =\eps^{\alpha-1}\norm{Qw_\ell}
 \le\eps^{\alpha-1}L.
\]
Because both $d$ and $d_0$ lie in $\Rdel$, so does $d-d_0$.  By the triangle inequality for the seminorm and \eqref{eq:rho-cont},
\begin{align*}
 \rho_\ell(d)
 &\le\rho_\ell(d_0)+\rho_\ell(d-d_0)\\
 &\le\eps^{\alpha-1}L
 +c_\ell\eps^\alpha\norm{P_+}\norm{w_\ell}\\
 &\le\eps^{\alpha-1}
 \bigl(1+c_\ell\norm{P_+}\eps\bigr)L.
\end{align*}
Taking the infimum over trajectories proves \eqref{eq:compat-lower}.  If $d\notin\cB_\ell$, \cref{lem:rho} gives $\rho_\ell(d)>0$, which implies \eqref{eq:compat-omega}.
\end{proof}

These estimates are lower bounds, not universal asymptotic equivalences.  Their role is to separate the two mechanisms and their natural scales.  The examples below show that each exponent can be attained in a natural model.

\section{Position--velocity benchmark}\label{sec:application}

We now return to the intermittent-actuation picture from the Introduction.  A NASA CubeSat docking study used a translational double-integrator model and treated minimum impulse bit as a dead-zone-like limitation; close-range rendezvous control allocation has also been formulated with an explicit minimum-impulse-bit constraint \citep{PeiEtAl2017,BotelhoRosaLemos2025}.  We use this setting only as motivation.  The model below is dimensionless and event-level, and its purpose is to make the compatibility theorem concrete in the simplest position--velocity system.

Let $\widetilde q$ and $\widetilde v$ be one relative position and velocity component.  Choose reference scales $q_*,v_*>0$ and $g_*=q_*/v_*$, and set $q=\widetilde q/q_*$ and $v=\widetilde v/v_*$.  In normalized variables take
\begin{equation}\label{eq:double-matrices}
 A=\begin{pmatrix}0&1\\0&0\end{pmatrix},\qquad B=e_2.
\end{equation}
Three correction opportunities have event weights
\begin{equation}\label{eq:double-masses}
 m_1=a>0,\qquad m_2=\varepsilon>0,\qquad m_3=b>0,
\end{equation}
and the middle event loses operational weight.  The Stieltjes event law is
\begin{equation}\label{eq:application-event-map}
 q_j=q_{j-1}+m_jv_{j-1},\qquad v_j=v_{j-1}+m_ju_j.
\end{equation}
These atomic weights are effective event weights, not calibrated pulse durations or impulses.  Equation \eqref{eq:application-event-map} is likewise an exact law of the reduced Stieltjes model, not a finite-duration constant-thrust propagator.  Although $A^2=0$ gives the same homogeneous factor $e^{mA}=I+mA$, a true finite-duration constant acceleration would also generate a position-control term of order $m^2u/2$.

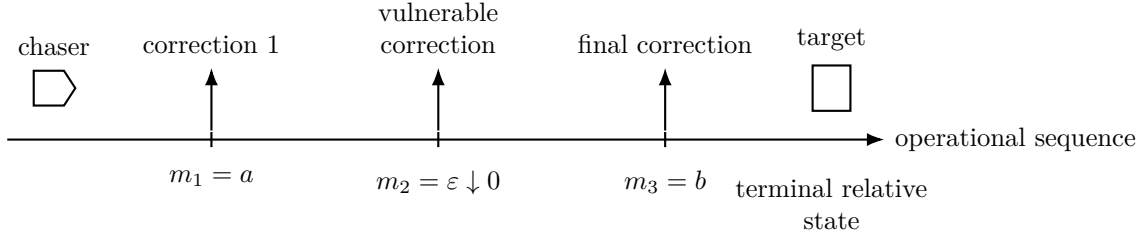
\begin{figure}[htbp]
	\centering
	\begin{tikzpicture}[x=1cm,y=1cm,>=Latex,font=\small]
		\draw[->,thick] (0,0) -- (11.6,0) node[right] {operational sequence};
		
		\draw[thick] (0.35,0.45) -- (0.75,0.45) -- (0.90,0.68) -- (0.75,0.91) -- (0.35,0.91) -- cycle;
		\node[above=2pt] at (0.62,0.91) {chaser};
		\draw[thick] (10.65,0.38) rectangle (11.15,0.98);
		\node[above=2pt] at (10.90,0.98) {target};
		\node[below=11pt,align=center] at (10.90,0) {terminal relative\\state};
		
		\draw[thick] (2.7,-0.10) -- (2.7,0.10);
		\draw[->,thick] (2.7,0.12) -- (2.7,0.95);
		\node[above=1pt] at (2.7,0.97) {correction 1};
		\node[below=5pt] at (2.7,-0.10) {$m_1=a$};
		
		\draw[thick] (5.7,-0.10) -- (5.7,0.10);
		\draw[->,thick] (5.7,0.12) -- (5.7,0.95);
		\node[above=1pt,align=center] at (5.7,0.97) {vulnerable\\correction};
		\node[below=5pt] at (5.7,-0.10) {$m_2=\varepsilon\downarrow0$};
		
		\draw[thick] (8.7,-0.10) -- (8.7,0.10);
		\draw[->,thick] (8.7,0.12) -- (8.7,0.95);
		\node[above=1pt] at (8.7,0.97) {final correction};
		\node[below=5pt] at (8.7,-0.10) {$m_3=b$};
	\end{tikzpicture}
	\caption{Schematic one-dimensional proximity maneuver with three effective correction opportunities.  The middle opportunity loses Stieltjes event weight while the first and last remain fixed.  The drawing is not to scale, and the event weights are atomic weights in the associated Stieltjes measure rather than calibrated pulse durations or impulse bits.}
	\label{fig:physical-timeline}
\end{figure}

At the vulnerable middle event,
\begin{equation}\label{eq:application-middle-jump}
 w_2
 =\frac{x_2-x_1}{\varepsilon^\alpha}
 =\varepsilon^{1-\alpha}\begin{pmatrix}v_1\\u_2\end{pmatrix}.
\end{equation}
Hence bounded cost with $\alpha>1$ requires $v_1\to0$.  This is exactly the abstract condition $Ax_1\in\operatorname{Ran}B$, because $Ax=(v,0)^\top$ and $\operatorname{Ran}B=\operatorname{span}\{e_2\}$.

Deleting Event 2 leaves endpoint columns $(ab,a)^\top$ and $(0,b)^\top$, whose determinant is $ab^2$.  Therefore
\begin{equation}\label{eq:double-Rdel}
 \mathcal R_{\rm del}=\mathbb R^2.
\end{equation}
Ordinary controllability is preserved.  On the other hand,
\[
 \mathcal K_1(A,B)=\operatorname{span}\{e_2\},\qquad
 A^{-1}(\operatorname{Ran}B)=\operatorname{span}\{e_1\},
\]
so \cref{thm:main} gives
\begin{equation}\label{eq:double-B}
 \mathcal B_2=\operatorname{span}\{e_2\}
 \subsetneq\mathbb R^2=\mathcal R_{\rm del}.
\end{equation}
The distinction is easy to misread.  Position is still controllable in the ordinary sense.  What fails is uniform boundedness of the event-normalized super-$g$-H\"older cost for terminal directions outside $\operatorname{span}\{e_2\}$ as the middle event vanishes.

\begin{figure}[htbp]
	\centering
	\begin{tikzpicture}[>=Latex,font=\small,x=0.9cm,y=0.9cm]
		\begin{scope}[yshift=0cm]
			\node[font=\bfseries] at (0,2.45) {(a) Deleted-event reachability};
			\fill[black!4] (-2.35,-1.85) rectangle (2.35,1.85);
			\draw[->,thick] (-2.55,0) -- (2.65,0) node[below right] {$q$};
			\draw[->,thick] (0,-1.95) -- (0,1.85) node[above right] {$v$};
			\fill (0,0) circle (1.2pt);
			\fill (1.45,0) circle (1.6pt);
			\node[above=4pt] at (1.45,0) {$e_1$};
			\fill (0,1.30) circle (1.6pt);
			\node[right=4pt] at (0,1.30) {$e_2$};
			\node[align=center] at (0,-1.30) {$\mathcal R_{\rm del}=\mathbb R^2$};
		\end{scope}
		
		\begin{scope}[yshift=-5.05cm]
			\node[font=\bfseries] at (0,2.45) {(b) Uniformly bounded-cost limit};
			\draw[->,thick] (-2.55,0) -- (2.65,0) node[below right] {$q$};
			\draw[->,thick] (0,-1.95) -- (0,1.85) node[above right] {$v$};
			\draw[line width=2.1pt] (0,-1.72) -- (0,1.72);
			\fill (0,0) circle (1.2pt);
			\fill (0,1.30) circle (1.6pt);
			\node[right=4pt] at (0,1.30) {$e_2$};
			\draw (1.45,0) circle (2.2pt);
			\node[above=5pt] at (1.45,0) {$e_1\notin\mathcal B_2$};
			\node[align=center] at (-1.18,-1.30) {$\mathcal B_2=\operatorname{span}\{e_2\}$};
		\end{scope}
	\end{tikzpicture}
	\caption{For the three-event double integrator with the middle event removed, the remaining two events still reach the whole position--velocity plane (top).  In the vanishing-event limit under a uniform super-$g$-H\"older budget, the intermediate compatibility condition leaves only the velocity direction $\mathcal B_2=\operatorname{span}\{e_2\}$ (bottom); the pure position direction $e_1$ is ordinarily reachable but not in the bounded-cost subspace.}
	\label{fig:reachability-geometry}
\end{figure}

Consider the normalized terminal command $d=e_1=(1,0)^\top$, a position correction with zero terminal velocity correction.  Put
\begin{equation}\label{eq:D-eps}
 D_\varepsilon^2=(b+\varepsilon)^2+b^2\varepsilon^2.
\end{equation}

\begin{proposition}[Exact small-event cost in the benchmark]\label{prop:double-exact-cost}
Let $a,b>0$ and $\alpha>1$.  If
\begin{align}
 \varepsilon^{\alpha-1}(b+\varepsilon)
 &\le a^\alpha D_\varepsilon,\label{eq:double-dominance-1}\\
 \varepsilon^{\alpha-1}(b+\varepsilon+b\varepsilon^2)\sqrt{b^2+1}
 &\le b^\alpha D_\varepsilon,\label{eq:double-dominance-2}
\end{align}
then
\begin{equation}\label{eq:double-exact}
 \Lambda_{\alpha,(a,\varepsilon,b)}(e_1)
 =\frac{\varepsilon^{1-\alpha}}{D_\varepsilon}.
\end{equation}
In particular, \eqref{eq:double-exact} holds for all sufficiently small $\varepsilon>0$.
\end{proposition}

\begin{proof}
For any feasible three-event trajectory reaching $e_1$, let $y=v_1$ be the velocity immediately before the vulnerable event and let $t=u_2$.  The first component of the event-2 constraint gives
\[
 w_2=\varepsilon^{1-\alpha}(y,t)^\top.
\]
Since the position increment across Event~2 is $\varepsilon y$ and the velocity entering Event~3 is $y+\varepsilon t$, the terminal position condition is
\begin{equation}\label{eq:double-linear-constraint}
 (b+\varepsilon)y+b\varepsilon t=1.
\end{equation}
Cauchy--Schwarz applied to \eqref{eq:double-linear-constraint} yields
\[
 \|w_2\|
 =\varepsilon^{1-\alpha}\sqrt{y^2+t^2}
 \ge \frac{\varepsilon^{1-\alpha}}{D_\varepsilon}.
\]
Hence the right-hand side of \eqref{eq:double-exact} is a lower bound for the global maximum cost.

Equality in the preceding Cauchy--Schwarz bound is attained by
\begin{equation}\label{eq:double-opt-yt}
 y=\frac{b+\varepsilon}{D_\varepsilon^2},
 \qquad
 t=\frac{b\varepsilon}{D_\varepsilon^2}.
\end{equation}
Choose
\begin{equation}\label{eq:opt-w12}
 w_1=\begin{pmatrix}0\\ \dfrac{b+\varepsilon}{a^\alpha D_\varepsilon^2}\end{pmatrix},
 \qquad
 w_2=\frac{\varepsilon^{1-\alpha}}{D_\varepsilon^2}
 \begin{pmatrix}b+\varepsilon\\ b\varepsilon\end{pmatrix}.
\end{equation}
Then $v_2=y+\varepsilon t=(b+\varepsilon+b\varepsilon^2)/D_\varepsilon^2$, and Event~3 is completed by
\begin{equation}\label{eq:opt-w3}
 w_3=\frac{v_2}{b^\alpha}\begin{pmatrix}b\\-1\end{pmatrix}.
\end{equation}
The first component of \eqref{eq:opt-w3} is exactly the drift constraint at Event~3, while the second enforces zero terminal velocity.  The corresponding block norms are
\[
 \|w_1\|=\frac{b+\varepsilon}{a^\alpha D_\varepsilon^2},\qquad
 \|w_2\|=\frac{\varepsilon^{1-\alpha}}{D_\varepsilon},\qquad
 \|w_3\|=\frac{(b+\varepsilon+b\varepsilon^2)\sqrt{b^2+1}}
 {b^\alpha D_\varepsilon^2}.
\]
Conditions \eqref{eq:double-dominance-1}--\eqref{eq:double-dominance-2} are precisely the inequalities $\|w_1\|\le\|w_2\|$ and $\|w_3\|\le\|w_2\|$.  Under them the constructed trajectory attains the lower bound, proving \eqref{eq:double-exact}.  Since $\alpha>1$, the left-hand sides of \eqref{eq:double-dominance-1}--\eqref{eq:double-dominance-2} tend to zero while their right-hand sides tend to positive limits as $\varepsilon\downarrow0$, so both inequalities hold for all sufficiently small $\varepsilon$.
\end{proof}

For the minimizing family above,
\begin{equation}\label{eq:actual-jump-norm}
 |u_2|=\frac{b\varepsilon}{D_\varepsilon^2}\longrightarrow0,
 \qquad
 \|\Delta x_2\|=\frac{\varepsilon}{D_\varepsilon}\longrightarrow0,
\end{equation}
whereas
\begin{equation}\label{eq:normalized-burden-diverge}
 \|w_2\|=\frac{\varepsilon^{1-\alpha}}{D_\varepsilon}\longrightarrow\infty,
 \qquad
 v_1\longrightarrow\frac1b.
\end{equation}
The singularity is therefore relative rather than absolute.  Neither the middle command nor the raw state jump explodes; in fact both quantities displayed in \eqref{eq:actual-jump-norm} tend to zero.  The problem is that the jump decays only at order $\varepsilon$, whereas supercritical regularity asks for the stronger scale $\varepsilon^\alpha$.  This is exactly the compatibility-loss mechanism.

\begin{figure}[htbp]
	\centering
	\begin{tikzpicture}
		\begin{loglogaxis}[
			width=0.86\textwidth,
			height=0.48\textwidth,
			xmin=1e-3,xmax=1e-1,
			ymin=1e-3,ymax=5e1,
			xlabel={degrading event weight $\varepsilon$},
			ylabel={magnitude},
			grid=both,
			legend pos=south west,
			legend cell align=left,
			tick label style={font=\small},
			label style={font=\small},
			legend style={font=\small,fill=white},
			]
			\addplot+[thick,no marks,domain=1e-3:1e-1,samples=180]
			{x/sqrt((1+x)^2+x^2)};
			\addlegendentry{$\|\Delta x_2\|$}
			
			\addplot+[thick,dashed,no marks,domain=1e-3:1e-1,samples=180]
			{x/((1+x)^2+x^2)};
			\addlegendentry{$|u_2|$}
			
			\addplot+[thick,dotted,no marks,domain=1e-3:1e-1,samples=180]
			{x^(-0.5)/sqrt((1+x)^2+x^2)};
			\addlegendentry{$\|w_2\|$, $\alpha=3/2$}
		\end{loglogaxis}
	\end{tikzpicture}
	\caption{What becomes singular in the exact optimal family for $a=b=1$ and $\alpha=3/2$, on the verified small-event interval $10^{-3}\le\varepsilon\le10^{-1}$.  The actual middle-event state jump $\|\Delta x_2\|$ and the command magnitude $|u_2|$ both vanish as the opportunity disappears, while the event-normalized state variation $\|w_2\|$ diverges.  The singularity is therefore relative to the shrinking event scale, not an exploding physical jump.}
	\label{fig:optimizer-structure}
\end{figure}
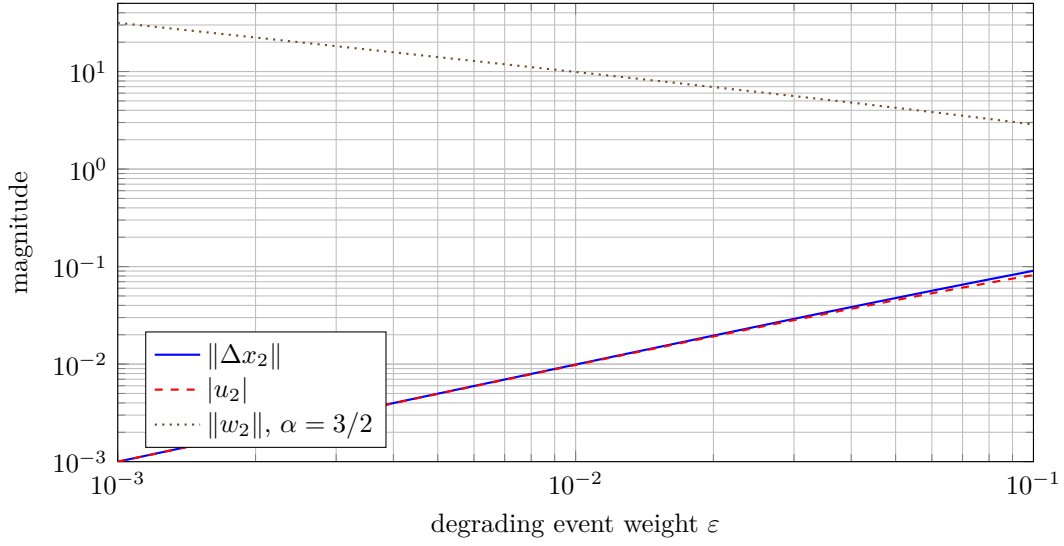

The cost has the asymptotic form
\begin{equation}\label{eq:double-asymptotic}
 \Lambda_{\alpha,(a,\varepsilon,b)}(e_1)
 \sim\frac{1}{b\varepsilon^{\alpha-1}}.
\end{equation}
Accordingly, a selected dimensionless regularity budget $L$ gives the small-event design scale
\begin{equation}\label{eq:budget-asymptotic}
 \varepsilon_{\rm crit}\sim(bL)^{-1/(\alpha-1)}.
\end{equation}
This scale has a limited but useful interpretation.  In the reduced model it marks when a chosen terminal command ceases to satisfy the prescribed regularity budget, even though ordinary reachability is unchanged.  It is not a fuel, saturation, or flight-safety threshold.  For the illustrative values $a=b=1$, $\alpha=3/2$, and $L=10$, the asymptotic estimate is $10^{-2}$, while the exact equality from \eqref{eq:double-exact} gives approximately $9.806\times10^{-3}$.

\begin{figure}[htbp]
	\centering
	\begin{tikzpicture}
		\begin{loglogaxis}[
			width=0.86\textwidth,
			height=0.50\textwidth,
			xmin=1e-3,xmax=1e-1,
			ymin=2e-1,ymax=2e3,
			xlabel={degrading event weight $\varepsilon$},
			ylabel={$\Lambda_{\alpha,(1,\varepsilon,1)}(e_1)$},
			grid=both,
			legend pos=north east,
			legend cell align=left,
			tick label style={font=\small},
			label style={font=\small},
			legend style={font=\small,fill=white},
			]
			\addplot+[thick,no marks,domain=1e-3:1e-1,samples=180]
			{x^(-0.25)/sqrt((1+x)^2+x^2)};
			\addlegendentry{$\alpha=1.25$}
			
			\addplot+[thick,dashed,no marks,domain=1e-3:1e-1,samples=180]
			{x^(-0.5)/sqrt((1+x)^2+x^2)};
			\addlegendentry{$\alpha=1.5$}
			
			\addplot+[thick,dotted,no marks,domain=1e-3:1e-1,samples=180]
			{x^(-1)/sqrt((1+x)^2+x^2)};
			\addlegendentry{$\alpha=2$}
			
			\addplot+[thin,dashdotted,no marks,domain=1e-3:1e-1] {10};
			\addlegendentry{illustrative budget $L=10$}
		\end{loglogaxis}
	\end{tikzpicture}
	\caption{Event-normalized regularity cost for the normalized terminal position correction $e_1$, with $a=b=1$, on $10^{-3}\le\varepsilon\le10^{-1}$.  For the three displayed exponents, the dominance conditions in Proposition~\ref{prop:double-exact-cost} hold throughout this interval, so the plotted curves use the exact small-event formula.  The horizontal line $L=10$ is an illustrative regularity budget; its intersections indicate where that selected budget is crossed, not a physical safety threshold.}
	\label{fig:cost-event-weight}
\end{figure}
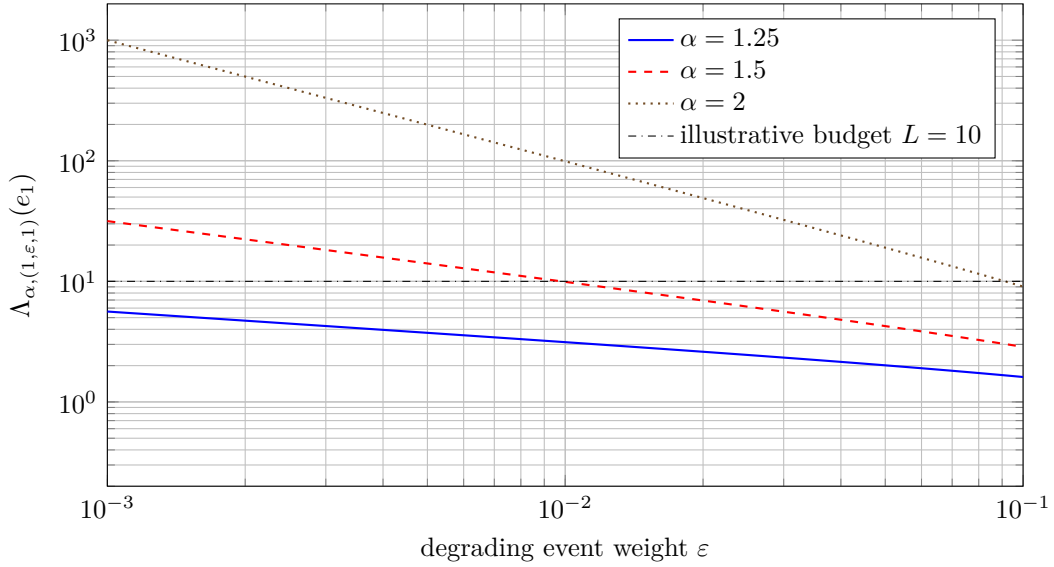

The figures use only formulas proved in the small-event regime.  The cost plots are therefore restricted to $10^{-3}\le\varepsilon\le10^{-1}$, an interval on which the displayed parameter choices satisfy the dominance conditions \eqref{eq:double-dominance-1}--\eqref{eq:double-dominance-2}.  The benchmark should be read as a one-dimensional analytical model.  Orbital coupling, saturation, uncertainty, and propulsion calibration are outside its scope, and $\Lambda$ measures event-normalized trajectory regularity rather than a standard engineering effort.
\section{Discussion and conclusions}\label{sec:discussion}

The paper began with a simple discrepancy.  In a finite-atomic Stieltjes system, ordinary reachability can remain unchanged while one positive event weight becomes arbitrarily small.  Once the trajectory is measured in a supercritical super-$g$-H\"older scale, that insensitivity disappears.  The vanishing event leaves a condition on the state reached just before it: the drift of that state must lie in the available input range.  Combining this compatibility condition with the directions reachable before and after the event gives the exact bounded-cost terminal space.

This is the point at which the present result goes beyond the individual classical ingredients.  Krylov reachability, predecessor constructions, minimum-amplitude control, and singular parametric optimization are all established theories \citep{BasileMarro1969,Cadzow1974,Kolev1975,Wonham1979,AvrachenkovEtAl2012,BorrelliBemporadMorari2017,RockafellarWets1998}.  Likewise, ordinary Stieltjes controllability and supercritical atomic regularity are available from earlier work \citep{SiWangFeckan2020,FernandezReachability2026,IlterDuruKoca2026}.  Here these ingredients meet in a singular fixed-target problem: the supercritical scaling selects an intermediate drift-cancellability condition, and that condition determines the terminal directions whose minimum normalized-jump cost stays bounded.  We do not claim novelty for predecessor geometry or singular optimization in isolation.

The two divergence mechanisms clarify why ordinary endpoint reachability is too coarse for this limit.  If the target disappears from the deleted-event reachable space, the vulnerable event must supply a terminal component of order one and the natural lower-bound scale is $\varepsilon^{-\alpha}$.  If the target remains ordinarily reachable but every limiting realization violates compatibility, the obstruction is subtler and the lower-bound scale becomes $\varepsilon^{1-\alpha}$.  On the bounded-cost subspace, neither mechanism occurs: the minimum costs converge to the deleted-event problem with the compatibility condition retained.

The position--velocity benchmark makes this distinction visible without changing the mathematical setting.  Deleting the middle event leaves the reduced double integrator controllable, yet a pure terminal position direction lies outside the bounded-cost space.  The exact minimizing family also prevents an overly physical interpretation of the divergence.  The vulnerable-event command and the raw state jump both tend to zero; what diverges is the jump relative to the smaller scale required by supercritical regularity.  The benchmark is therefore an analytical illustration of the singular geometry, not a calibrated rendezvous or propulsion model.

The present theory is intentionally finite-dimensional and exact: $A$ and $B$ are constant, the clock has finitely many atoms, one event weight vanishes, and $\alpha>1$.  Several extensions are natural, including several vanishing events, bounded controls, uncertain event weights, time-dependent matrices, and multi-axis relative-motion models.  The central question in each case is the same: which intermediate compatibility conditions survive the singular limit, and how do they reshape the terminal reachable geometry?

\section*{Use of Generative-AI Tools Declaration}
During the preparation of this manuscript, the authors used AI-assisted tools to improve the clarity, structure, and presentation of the text, and to support the verification of mathematical arguments and results. All mathematical content, proofs, and conclusions were independently reviewed and validated by the authors, who take full responsibility for the accuracy, originality, and integrity of the work.

\end{document}